\documentclass[12pt,a4paper]{article}
\usepackage{graphicx,url}
\usepackage{amssymb,amsmath, amsthm}
\usepackage{epstopdf}
\usepackage{natbib}
\usepackage{makeidx}
\usepackage{color}

\newtheorem{theorem}{Theorem}

\newtheorem{assumption}{Assumption}
\newtheorem{remark}{Remark}

\title{On penalized estimation for dynamical systems with small noise}
\author{Alessandro De Gregorio\footnote{
Department of Statistical Sciences,
P.le Aldo Moro 5,
 00185- Rome, Italy - alessandro.degregorio@uniroma1.it}\quad
 Stefano M. Iacus\footnote{Department of Economics, Management and Quantitative Methods,
 Via Conservatorio 7, 20122 - Milan, Italy - stefano.iacus@unimi.it}}

\makeindex

\def\de{{\rm d}}
\def\ve{{\varepsilon}}

\begin{document}

\maketitle

\begin{abstract}
We consider a dynamical system with small noise for which the drift is parametrized by a finite dimensional parameter. For this model we consider minimum distance estimation from continuous time observations under  $l^p$-penalty imposed on the parameters in the spirit of the Lasso approach with the aim of  simultaneous estimation and model selection. We study the  consistency  and the asymptotic distribution of these Lasso-type estimators for different values of $p$. For $p=1$ we also consider the adaptive version of the Lasso estimator and establish its oracle properties.

\end{abstract}

\medskip
{\bf  Keywords:} dynamical systems, lasso estimation, model selection, inference for stochastic processes,
diffusion-type processes, oracle properties.


\section{Introduction}
Usually ordinary differential equation models are the result of averaging and/or neglecting some details of an original system without modeling a complex system with a huge number of degrees of freedom or tuning parameters. Introducing noise is therefore a way to approximate closer the reality of observable complex systems. It is then natural to think of the noise as small, for example when one is considering the dynamics of macroscopic quantities, i.e. averages of quantities of interest over a whole population or in the case of signal that travels through a perturbed medium, etcetera.

Dynamical systems with small perturbations have been indeed widely studied in  \citet{Azencott82} and \citet{FreidlinWentzell98}. Applications of small diffusion processes to mathematical finance and option pricing have been considered in \citet{Yoshida92b}, \citet{Kunitomo01}, \citet{TakaYoshi04}, \citet{UchidaYoshida04} and references therein. Examples from biology and life sciences include \citet{Murray02}, \citet{Bressloff14}, \citet{Ermentrout10}.

Model selection is an important aspect in the above applied fields although sometimes neglected. What occurs for dynamical systems with small noise, is not so different from what happens in ordinary least squares (OLS) model estimation. Indeed, linear regression models are used extensively by many practitioners but, once estimated,  these models are useful as long as the set of parameter (or covariates) is correctly specified. Therefore, the  model selection step is an important part of the analysis. 

To introduce the idea of Lasso-type estimation we begin with linear models and OLS. In this framework model selection occurs when some of the  regression parameters  are estimated as  zero. Different models are compared in terms of information criteria like AIC/BIC or hypotheses testing.
The advantage of the Lasso-type approach over AIC/BIC is that statistical models do not need to be nested but one can rather construct  a single large parametric model merging two orthogonal models and let the selection method to choose one of the two models \citep{Caner09}.

Variable selection becomes particularly important when the true underlying model has a sparse representation. Identifying correctly significant predictors will improve the prediction performance of the fitted model \citep[for an overview of feature selection see][]{FanLi06}.

Considered the linear regression model $Y_i = x^T_i \beta + \ve_i$, with $x_i$ a vector of covariates, $\beta$ a vector of $q>0$ parameters and  $\ve_i$ i.i.d. Gaussian random variables. \citet{Knight00}  proposed the following  $l^p$-penalized estimator for $\beta$
\begin{equation}
\hat\beta_n := \arg\min_u \left(\sum_{i=1}^n (Y_i - x^T_i u)^2 + \lambda_n \sum_{j=1}^q |u_j|^p\right)
\label{eq1}
\end{equation}
for some $p>0$ and $\lambda_n\to 0$ as $n\to\infty$. The family of estimators $\hat\beta_n$ solution to \eqref{eq1}  are a generalization of the Ridge estimators which correspond to the case $p=2$ \citep{Efron}. The original Lasso estimators \citep{Tib96} are obtained setting $p=1$ while OLS is the case $p=0$, not considered here.
The link between Lasso-type estimation and model selection is also due to the fact that, in the limit as $p\to 0$, this procedure approximate the AIC or BIC selection methods, i.e.
$$\lim_{p\to 0} \sum_{j=1}^q |u_j|^p = \sum_{j=1}^q {\bf 1}_{\{u_j\neq 0\}}$$
which amounts to the number of non-null parameters in the model. Here ${\bf 1}_A$ the indicator function for set $A$.

As said, the estimators solutions to \eqref{eq1} are attractive because with them it is possible to perform estimation and model selection in a single step, i.e. the procedure does not need to estimate different models  and compare them later with information criteria as the dimension of the space of the parameters does no change, just some of the components of the vector $\beta_j^*$ are assumed to be zero.
In non-linear models a preliminary simple reparametrization (e.g. $\beta \mapsto \beta'-\beta$) is needed to interpret this approach in terms of model selection.

In this work, we extend the problem in \eqref{eq1} to the class of diffusion processes with small noise solution to the stochastic differential equation $\de X_t = S_t(\theta, X) \de t + \ve  \de W_t$, $t \in [0,T]$, by  replacing the least squares estimation with the minimum distance estimation. The asymptotic is considered as $\ve\to 0$ for fixed $0<T<\infty$ with $\theta\in\Theta\subset \mathbb R^q$ a $q$-dimensional parameter.

Since the seminal works of  \citet{Kuto84, Kuto91, Kuto94} and \citet{Yoshida92a}, statistical inference for continuously observed small diffusion processes is well developed today \citep[see, e.g.,][]{KutoPhil94, Iacus00, IacusKuto01, Yoshida03, UchidaYoshida04a} but the Lasso problem has not been considered so far. Although here we consider only continuous time observations, it is worth mentioning that  there is also a growing literature on parametric inference for discretely observed small diffusion processes \citep[see][]{GenonCatalot90, Laredo90, Sorensen00, Sorensen12, SorenUchida03, Uchida03, Uchida04, Uchida06, Uchida08, GloterSorensen09, Guy14} to which this Lasso problem can be extended. Adaptive Lasso-type estimation for ergodic diffusion processes sampled at discrete time has been studied in \citet{DegIac12} while for continuous time ergodic diffusion processes shrinkage estimation has been considered in \citet{Nkurunziza12}.

This paper is organized as follows. In Section \ref{sec1} we introduce the model, the assumptions and the statement of the problem. In Section \ref{sec2} we study the consistency  of the estimators and derive their  asymptotic distribution for different values of $p$. For $p= 1$ we also consider the case of  adaptive Lasso estimation that is meant to control  asymptotic bias. For the adaptive estimator, we are also able to prove that it represents an oracle procedure.

\section{The Lasso-type problem for dynamical systems with small noise}\label{sec1}

Let us assume that on the probability space $(\Omega,\mathcal F, P),$ with the filtration $\{\mathcal F_t,0\leq t\leq T\}$ (where each $\mathcal F_t,0\leq t\leq T,$ is augmented by sets from $\mathcal F$ having zero $P$-measure), is given a Wiener process $\{W_t,\mathcal F_t,0\leq t\leq T\}.$
Let $X=\{X_t,0 \leq t \leq T\}$ be a real valued diffusion-type process solution to the following stochastic differential equation
\begin{equation}
\de X_t = S_t(\theta, X) \de t + \ve  \de W_t,\quad \varepsilon\in(0,1],
\label{eq:mod1}
\end{equation}
with non random initial condition $X_0=x_0$. The parameter $\theta\in\Theta\subset \mathbb R^q,$ where $\Theta$ is a bounded, open and convex set, supposed to be unknown. Let $(C[0,T],\mathcal B[0,T])$ be the measurable space of continuous functions $x_t$ on $[0,T]$ with $\sigma$-algebra $\mathcal B[0,T]=\sigma\{x_t,0\leq t\leq T\}.$ $P_\theta^{(\varepsilon)}$ denotes the law induced by the process $X$ in $(C[0,T],\mathcal B[0,T])$ when the true parameter is $\theta$. 
We denote by $u = (u_1, \ldots, u_q)^T$ the (transposed) vector $u\in \mathbb R^q$ and the true value of $\theta$ by $\theta^*$.
Let $|| \cdot || = ||\cdot ||_{L_2(\mu)}$  be the $L_2$-norm with respect to some finite measure $\mu$ on $[0,T]$, i.e.
$$
|| f ||^2 = \int_0^T f^2(x) \mu(\de x).
$$
We suppose that the trend coefficient in \eqref{eq:mod1} is of integral type, i.e.
$$
S_t(\theta, X) = V(\theta, t,X) + \int_0^t K(\theta, t,s,X_s)\de s,
$$
where $V(\theta,t,x)$ and $K(\theta, t, s, x)$ are known measurable, non-anticipative functionals such that \eqref{eq:mod1}  has a strong unique solution. 
For example, the usual conditions (1.34) and (1.35) in \citet{Kuto94} and Theorem 4.6 in \citet{LipSh01} about Lipschitz behaviour and linear growing are sufficient; i.e.
\begin{assumption}\label{exun}
For all $t\in[0,T],$ $\theta\in\Theta$ and $X_t,Y_t\in C[0,T]$
\begin{align*}
&|V(\theta,t,X_t)-V(\theta,t,Y_t)|+|K(\theta, t,s,X_t)-K(\theta, t,s,Y_t)|\\
&\leq L_1\int_0^t|X_s-Y_s|\de K_s+L_2|X_t-Y_t|,\\
&|V(\theta,t,X_t)|+|K(\theta, t,s,X_t)|\leq L_1\int_0^t(1+|X_s|)\de K_s+L_2(1+|X_t|),
\end{align*}
where $L_1$ and $L_2$ are positive constants and $K_s$ is a nondecreasing right-continuous function, $0\leq K_t\leq K_0, K_0>0.$
\end{assumption}
Assumption \ref{exun} implies that all the probability measures $P_\theta^{(\varepsilon)},\theta\in\Theta,$ are equivalent (see Theorem 7.7 in \citet{LipSh01}).
 The asymptotic in this model is considered as $\ve\to 0$ and $0<T<\infty$ fixed.

We will also write $x(\theta) = x_t(\theta)$ to denote  the limiting dynamical system satisfying the integro-differential equation
$$
\frac{\de x_t}{\de t} = V(\theta, t, x_t) + \int_0^t K(\theta, t, s, x_s) \de s, \quad x_0.
$$

We assume that, for all $0\leq t\leq T$ and for each $\theta\in\Theta,$ the random element $X_t$ and $x_t(\theta)$ belong to $L_2(\mu).$ 

Let $x_t^{(1)} =\{ x_t^{(1)}(\theta^*),0\leq t\leq T\}$ be  the Gaussian process 
solution to
\begin{equation}
\label{eqx}
\de x_t^{(1)} = \left( V_x(\theta^*, t, x_t(\theta^*)) x_t^{(1)} + \int_0^t K_x(\theta_0, t, s, x_s(\theta^*)) x_s^{(1)} \de s\right)\de t + \de W_t, 
\end{equation}
$0 \leq t \leq T$, $x_0^{(1)} = 0$, where $V_x(\theta, t, x)=\frac{\partial}{\partial x} V(\theta, t, x)$ and $K_x(\theta, t, s, x)=\frac{\partial}{\partial x} K(\theta, t, s, x).$ The process $x_t^{(1)}$ plays a central role in the definition of the asymptotic distribution of the estimators in the theory of dynamical systems with small noise.
We need in addition the following assumptions. 
\begin{assumption}
\label{ass1}
The stochastic process $X$ is differentiable in $\ve$ at the point $\ve=0$ in the following sense: for all $\nu>0$
$$
\lim_{\ve\to 0} P_{\theta^*}^{(\varepsilon)}\left(  || \ve^{-1}(X-x) - x^{(1)}|| > \nu\right)=0
$$
where $x^{(1)} = \{ x^{(1)}_t, 0\leq t \leq T\}$ is from \eqref{eqx}. 
\end{assumption}

We further denote by $\dot x_t(\theta)$ the $q$-dimensional vector of partial derivatives of $x_t(\theta)$ with respect to $\theta_j$, $j=1, \ldots, q$, i.e., $\dot x_t(\theta)=(\frac{\partial}{\partial \theta_1} x_t(\theta),\ldots, \frac{\partial}{\partial \theta_q} x_t(\theta))^T$, and $\dot x_t(\theta^*)$ satisfies the systems of equations
\begin{align*}
\frac{\de \dot x_t(\theta^*)}{\de t}&=[V_x(\theta^*, t, x_t(\theta^*))  \dot x_t(\theta^*)+ \dot V(\theta^*, t, x_t(\theta^*))\\
&\quad+\int_0^t(\dot K(\theta^*, t, s, x_s(\theta^*))+K_x(\theta_0, t, s, x_s(\theta^*)) \dot x_s(\theta^*) )\de s]\de t,\quad \dot x_0(\theta^*)=0,
\end{align*}
where the point corresponds to the differentiation on $\theta;$ i.e.
$$\dot V(\theta, t, x_t(\theta))=\left(\frac{\partial}{\partial \theta_1}V(\theta, t, x_t(\theta)),...,\frac{\partial}{\partial \theta_q}V(\theta, t, x_t(\theta))\right)^T.$$

\begin{assumption}
\label{ass2}
The deterministic dynamical system $x_t(\theta)$ is differentiable in $\theta$ at the point $\theta^*$ in $L_2(\mu)$; i.e.
$$
||x(\theta^*+h) - x(\theta^*) - h^T  \dot x(\theta^*))|| = o(|h|)
$$
where $h\in \mathbb R^q$.
\end{assumption}
\begin{assumption}
\label{ass3}
The  matrix
$$\mathcal I(\theta^*) = \int_0^T \dot x_t(\theta^*) \dot x_t^T(\theta^*) \mu(\de t)$$
is positive definite and nonsingular.
\end{assumption}

\subsection{The Lasso-type estimator}


We introduce a constrained minimum distance estimator for $\theta$ for the model in \eqref{eq:mod1}.
The asymptotic properties of unconstrained the minimum distance estimators in the i.i.d. framework have been established in \citet{Millar83, Millar84}. Later \citet{Kuto91, Kuto94} and \citet{KutoPhil94}
 studied in details the properties of such estimators for diffusion processes with small noise. Information criteria for this model have been studied in \citet{UchidaYoshida04a}, while here we study the Lasso-type approach.

To define the Lasso-type estimator the following  penalized contrast function has to be considered
\begin{equation}
\label{eq:2}
Z_\ve (u) =  || X - x(u)|| +  \lambda_\ve \sum_{j=1}^q |u_j|^p,
\end{equation}
$p>0$,  $u \in \Theta$ and $\lambda_\ve>0$ a real sequence.
In analogy to \eqref{eq1}, we introduce the Lasso-type estimator $\hat\theta^\varepsilon:C[0,T]\to \bar\Theta$ for $\theta$, defined as 
\begin{equation}
\hat\theta^\ve = \arg \min_{\theta\in\bar\Theta}  Z_\ve(\theta),
\label{eq:est}
\end{equation}
where $\bar\Theta$ is the closure od $\Theta$.

The following example explains well the spirit of the Lasso procedure. We consider a linear small diffusion-type process $X$ given by
$$\de X_t = \sum_{j=1}^q \theta_j A_j(t,X)\de t + \ve  \de W_t,\quad 0\leq t\leq T.$$
By applying the estimator \eqref{eq:est}, some parameters $\theta_j$ will be set equal to 0 and this implies a simultaneous estimation and selection of the model.

\section{Asymptotic properties of the estimator}\label{sec2}
The additional $l_p$-penalization term in the contrast function \eqref{eq:2} modifies the traditional properties of the minimum distance estimator. The analysis should be performed for the different values of $p$ which change the convexity of the penalty term. 
\subsection{Consistency of the estimator}
Let us introduce the following functions

$$
g^\ve_{\theta^*}(\nu) = \inf_{|\theta-\theta^*|\geq \nu} \left\{||x(\theta)-x(\theta^*)|| + \lambda_\ve \sum_{j=1}^q |\theta_j|^p\right\},$$

$$
h^\ve_{\theta^*}(\nu) = \inf_{|\theta-\theta^*|<\nu} \left\{||x(\theta)-x(\theta^*)|| + \lambda_\ve \sum_{j=1}^q |\theta_j|^p\right\}$$ where $|\theta-\theta^*|\geq\nu(<\nu)$ is to be intended component wise, for all $\nu>0$. 
We need the following identifiability-type condition.
\begin{assumption}\label{sep}
For every $\nu>0,$ we assume that
$$g^\ve_{\theta^*}(\nu) >h^\ve_{\theta^*}(\nu).$$
\end{assumption} 

\begin{theorem}\label{cons}

Let Assumption \ref{exun} and Assumption \ref{sep} be fulfilled and $\lambda_\varepsilon=O(\varepsilon)$ as $\ve \to 0$. 
$\hat\theta^\ve$ in \eqref{eq:est} is a uniformly consistent estimator of $\theta^*$; i.e. for any $\nu>0$
$$\lim_{\ve\to 0}\sup_{\theta^*\in\Theta}P_{\theta^*}^{(\varepsilon)}\left(
 |\hat\theta^\ve-\theta^*|\geq \nu
\right)=0.$$

\end{theorem}
\begin{proof}
By definition of $\hat\theta^\ve,$ for any $\nu>0,$ we have that
$$
\left\{
\omega : |\hat\theta^\ve-\theta^*|\geq \nu
\right\} =
\left\{
\omega : \inf_{|\theta-\theta^*|< \nu} Z_\ve(\theta) > \inf_{|\theta-\theta^*|\geq \nu} Z_\ve(\theta)
\right\}
$$
Moreover,
$$
\begin{aligned}
Z_\ve(\theta)  & \leq
 || X - x(\theta^*)|| + || x(\theta) - x(\theta^*)|| + \lambda_\ve \sum_{j=1}^q |\theta_j|^p,\\
Z_\ve(\theta) &\geq 
 || x(\theta) - x(\theta^*)|| - || X - x(\theta^*)|| + \lambda_\ve \sum_{j=1}^q |\theta_j|^p.
\end{aligned}
$$
Then, from the above inequality, we get
$$
\begin{aligned}
P_{\theta^*}^{(\varepsilon)}\left(
 |\hat\theta^\ve-\theta^*|\geq \nu
\right)
&= 
P_{\theta^*}^{(\varepsilon)}\left(
\inf_{|\theta-\theta^*|< \nu} Z_\ve(\theta) > \inf_{|\theta-\theta^*|\geq \nu} Z_\ve(\theta)
\right) \\
&\leq
P_{\theta^*}^{(\varepsilon)}\biggl(|| X - x(\theta^*)|| +\frac{h^\ve_{\theta^*}(\nu)}{2}  >  \frac{g^\ve_{\theta^*}(\nu) }{2} \biggr) 
\end{aligned}
$$
Since (see Lemma 1.13, in \citet{Kuto94}) 
$$||X_t-x_t(\theta^*)||\leq C\varepsilon \sup_{0\leq t\leq T} |W_t|,\quad P_{\theta^*}^{(\varepsilon)}-\text{a.s.},$$
where $C=C(L_1,L_2,K_0,T)$ is a positive constant, under Assumption \ref{sep}, we get
$$
\begin{aligned}
\sup_{\theta^*\in\Theta}P_{\theta^*}^{(\varepsilon)}\left(
 |\hat\theta^\ve-\theta^*|\geq \nu
\right)&
\leq
P_{\theta^*}^{(\varepsilon)}\left(C \ve \sup_{0\leq t\leq T}|W_t| > 
 \frac12\inf_{\theta^*\in\Theta}\{g^\ve_{\theta^*}(\nu)-h^\ve_{\theta^*}(\nu)\}
\right)\\
&\leq2\exp\left\{- \frac{(\inf_{\theta^*\in\Theta}\{g^\ve_{\theta^*}(\nu)-h^\ve_{\theta^*}(\nu)\})^2}{8TC^2\ve^2}\right\} \to 0.
\end{aligned}
$$
In the above we made use of the following estimate for $N>0$
$$P\left(\sup_{0\leq t \leq T}|W_t| > N\right)\leq 4P\left(W_T> N\right)\leq 2e^{-\frac{N^2}{2T}},
$$
see e.g. \citet{Kuto94}, and observed that
$$g^\ve_{\theta^*}(\nu)-h^\ve_{\theta^*}(\nu)\to\inf_{|\theta-\theta^*|\geq \nu} ||x(\theta)-x(\theta^*)||>0 ,\quad \ve\to 0.$$
\end{proof}
From the proof of the consistency of the estimator $\hat\theta^\ve$ is it clear that the speed of the convergence depends on the speed of $\lambda_\ve$. The speed of $\lambda_\ve$ also affects the asymptotic distribution of the estimator.

\begin{remark}
It is possible to define other types of Lasso-type estimators modifying the metric in \eqref{eq:2}; i.e. by considering, for instance, the sup-norm and the $L_1$-norm. Hence, if $\{X_t,0\leq t\leq T\}$ and $\{x_t(\theta),0\leq t\leq T\},\theta\in\Theta,$ are elements of the space $C[0,T]$ and $L_1(\mu),$ respectively, we can introduce the corresponding Lasso estimator 
\begin{equation*}
\check\theta^\ve=\arg\min_{\theta\in\bar \Theta}\left\{\sup_{0\leq t\leq T}| X_t - x_t(\theta)| +  \lambda_\ve \sum_{j=1}^q |u_j|^p\right\}
\end{equation*}
and
\begin{equation*}
\breve\theta^\ve=\arg\min_{\theta\in\bar \Theta}\left\{\int_0^T| X_t - x_t(\theta)|\mu(\de t) +  \lambda_\ve \sum_{j=1}^q |u_j|^p\right\}.
\end{equation*}
The estimators $\check\theta^\ve$ and $\breve\theta^\ve$ are uniformly consistent and the proof follows by the same steps adopted to prove Theorem \ref{cons}.
\end{remark}

\subsection{Asymptotic distribution of the estimator}
In order to study the asymptotic distribution of the Lasso-type estimator we need to distinguish the different cases for $p$. We start with the case of  $p\geq 1$. We denote by ``$\to_d$'' the convergence in distribution and we denote by $\zeta$ the following Gaussian random vector
\begin{equation}
\label{eq:zeta}
\zeta =\int_0^T x^{(1)}_t (\theta^*)\dot x_t(\theta^*) \mu(\de t);
\end{equation}
i.e. $\zeta\sim N_q({\bf 0}, \sigma^2)$ where 
$$\sigma^2:=\int_0^T \int_0^T\dot x_t(\theta^*)\dot x_s(\theta^*)^T E [x^{(1)}_t (\theta^*)x^{(1)}_s (\theta^*)]\mu(\de t) \mu(\de s),$$
see also Lemma 2.13 in \citet{Kuto94}. The next two theorems have been inspired from Theorem 2 and Theorem 3 in \citet{Knight00}.

 \begin{theorem}
\label{eq:th2}
Let Assumptions \ref{exun}--\ref{sep} hold, $\zeta$ defined as in \eqref{eq:zeta},
 $p\geq 1$ and $\ve^{-1}\lambda_\ve\to\lambda_0\geq 0$. Then
$$
\ve^{-1}(\hat\theta^\ve - \theta^*) \to_d \arg\min_{u} V(u)
$$
where
$$V(u) = -2u^T \zeta + u^T \mathcal I(\theta^*) u +\lambda_0 \sum_{j=1}^q u_j {\rm sgn}(\theta_j^*)|\theta_j^*|^{p -1}
$$
for $p>1$ and
$$V(u) = -2u^T \zeta + u^T \mathcal I(\theta^*) u +
\lambda_0  \sum_{j=1}^q \left(|u_j| {\bf 1}_{\{\theta^*_j=0\}} +
u_j {\rm sgn}(\theta_j^*)|\theta_j^*| {\bf 1}_{\{\theta^*_j\neq 0\}}
\right)
$$
if $p=1$. 
\end{theorem}
\begin{proof}
Let $u\in\mathbb R^q$ and introduce the random function
\begin{equation}\label{eq:V}
V_\ve(u) =\frac{1}{\ve^2} \left(||X - x(\theta^*+\ve u)||^2 - ||X-x(\theta^*)||^2 + 
\lambda_\ve \sum_{j=1}^q \left\{ |\theta_j^* + \ve u_j|^p - |\theta_j^*|^p\right\}\right),
\end{equation}
which is minimized at the point $u=\ve^{-1}(\hat\theta^\ve-\theta^*)$ by definition of $\hat\theta^\ve$.
By exploiting Assumption \ref{ass1}--\ref{ass3}, we get
\begin{align}\label{eq:conv1}
&\frac{1}{\ve^2} \left\{||X - x(\theta^*+\ve u)||^2 - ||X-x(\theta^*)||^2\right\}\notag	\\
&=\frac{1}{\ve^2} \left\{||X - x(\theta^*) - \ve u^T \dot x(\theta^*)||^2 - ||X-x(\theta^*)||^2\right\} + o_\ve(1)\notag\\
&=u^T ||\dot x(\theta^*)||^2 u -2 u^T ||\ve^{-1}(X-x(\theta^*)) \dot x(\theta^*)||+ o_\ve(1)\notag\\
&\underset{\ve\to 0}{\stackrel{P_{\theta^*}^{(\ve)}}{\longrightarrow}} u^T \mathcal I(\theta^*) u -2 u^T \zeta, 
\end{align}
where $\stackrel{P_{\theta^*}^{(\ve)}}{\longrightarrow}$ stands for the convergence in probability and $\zeta$ is from \eqref{eq:zeta}.
For the term in \eqref{eq:V}
\begin{align*}
\frac{\lambda_\ve}{\ve^2}  \sum_{j=1}^q \left\{ |\theta_j^* + \ve u_j|^p - |\theta_j^*|^p\right\}
\end{align*}
we have to distinguish the case $p=1$ and $p>1$.
Let $\gamma>1$, then
\begin{align}\label{eq:conv2}
&\frac{\lambda_\ve}{\ve^2}  \sum_{j=1}^q \left\{ |\theta_j^* + \ve u_j|^p - |\theta_j^*|^p\right\}\notag\\
&=
\frac{\lambda_\ve}{\ve}  \sum_{j=1}^q u_j\frac{ |\theta_j^* + \ve u_j|^p - |\theta_j^*|^p}{\ve u_j}
\underset{\ve\to 0}{\longrightarrow} \lambda_0  \sum_{j=1}^p u_j {\rm sgn}(\theta^*_j) |\theta_j^*|^{p-1}
\end{align}
If $p=1$, then by similar arguments, we have
\begin{align}\label{eq:conv3}
\frac{\lambda_\ve}{\ve^2}  \sum_{j=1}^q \left\{ |\theta_j^* + \ve u_j| - |\theta_j^*|\right\}\underset{\ve\to 0}{\longrightarrow}
\lambda_0  \sum_{j=1}^q \left(|u_j| {\bf 1}_{\{\theta^*_j=0\}} +
u_j {\rm sgn}(\theta_j^*)|\theta_j^*| {\bf 1}_{\{\theta^*_j\neq 0\}}
\right).
\end{align}
Notice that $V_\ve(u)$ is not convex in $u$ and then we have to consider the convergence in distribution on the topology induced by the uniform metric on compact sets; i.e. we deal with the convergence in distribution of $V_\ve(u)$ on the space of the continuous functions topologized by the distance $\rho(y_1,y_2)=\sup_{u\in K}|y_1(u)-y_2(u)|,$ where $K$ is a compact subset of $\mathbb R^d.$ From \eqref{eq:conv1}, \eqref{eq:conv2} and \eqref{eq:conv3} follows the convergence of the finite-dimensional distributions
$$(V_\ve(u_1),...,V_\ve(u_k))\to_d(V(u_1),...,V(u_k))$$
for any $u_i\in\mathbb R^d,i=1,...,k.$ The tightness of $V_\ve(u)$ is implied by  
$$\sup_{\ve\in(0,1]}E \left[\sup_{u\in K}\left|\frac{\de}{\de u} V_\ve (u)\right|\right]<\infty$$
which follows from the regularity conditions on $\{x_t(\theta),0\leq t\leq T\}.$ Indeed it is not hard to prove that
$$\lim_{h\to 0}\limsup_{\ve\to 0}E[w(V_\ve(u),h)\wedge1]\leq \lim_{h\to 0}h \sup_{\ve\in(0,1]}E \left[\sup_{u\in K}\left|\frac{\de}{\de u} V_\ve (u)\right|\right]=0,$$
where $w(y,h)=\sup\{\rho(y(u),y(v)):|u-v|\leq h\},$ with $y$ continuous function on compact sets and $h>0.$ Therefore by Theorem 16.5 in \citet{Kall01}, we conclude that
$$V_\ve(u)\to_d V(u)$$
uniformly on $u.$
Since $\arg\min_u V(u)$ is unique ($P_{\theta^*}^{(\ve)}-$a.s.), to prove that
$$
\arg\min V_\ve =\ve^{-1}(\hat\theta^\ve - \theta^*) \to_d \arg\min V,
$$
we can use Theorem 2.7 in \citet{KimPollard90}. Hence, it is sufficient to show that $\arg\min_uV_\ve(u)=O_{P_{\theta^*}^{(\ve)}}(1).$ We observe that
\begin{align*}
V_\ve(u)&=V_\ve^l(u)+o_\ve(1) \end{align*}
where 
\begin{align*}
V_\ve^l(u)&=\frac{1}{\ve^2} \Bigg\{u^T ||\dot x(\theta^*)||^2 u -2 u^T ||\ve^{-1}(X-x(\theta^*)) \dot x(\theta^*)||\\
&\quad+
\lambda_\ve \sum_{j=1}^q \left\{ |\theta_j^* + \ve u_j|^p - |\theta_j^*|^p\right\}\Bigg\} \end{align*}
 is a convex function. Since for each $a\in\mathbb R$ and $\delta>0,$ there exists a compact set $K_{a,\delta}$ such that (see, \citet{Knight99})
 $$\limsup_{\ve\to 0}P_{\theta^*}^{(\ve)}\left(\inf_{u\notin K_{a,\delta}} V_\ve(u)\leq a	\right)\leq \delta,$$
 then $\arg\min_uV_\ve(u)=O_{P_{\theta^*}^{(\ve)}}(1).$ 
\end{proof}

In the case $0<p<1$  the convexity argument cannot be applied, moreover, some rate of convergence must be imposed on  the sequence $\lambda_\ve$.

\begin{theorem}\label{th3}
Let Assumptions \ref{exun}--\ref{ass3} hold, $\zeta$ defined as in \eqref{eq:zeta},
 $0<p<1$ and $\lambda_\ve/\ve^{2-p} \to\lambda_0\geq 0$. Then
$$
\ve^{-1}(\hat\theta^\ve - \theta^*) \to_d \arg\min_{u} V(u)
$$
where
$$V(u) = -2u^T \zeta + u^T \mathcal I(\theta^*) u +
\lambda_0  \sum_{j=1}^q |u_j|^p {\bf 1}_{\{\theta^*_j=0\}}.
$$
\end{theorem}
\begin{proof}
We proceed analogously to the proof of Theorem \ref{eq:th2}.
As before we start with $V_\ve(u)$ from \eqref{eq:V}.
The first part of the expression in $V_\ve(u)$  converges in distribution to $-2u^T \zeta + u^T \mathcal I(\theta^*) u$ 
as in Theorem \ref{eq:th2}. For the second term, we need to distinguish the two cases $\theta_k^*=0$ or $\theta_k^*\neq0$. By assumptions we have that $\lambda_\ve/\ve^{2-p} \to \lambda_0$ and hence necessarily $\lambda_\ve/\ve \to 0$. 

Consider first the case  $\theta_k^*\neq0$. We have that
$$
\frac{\lambda_\ve}{\ve} u_k \left( \frac{|\theta^*_k + \ve u_k|^p - |\theta_k^*|^p}{\ve u_k} \right)\to 0.
$$
Conversely, if  $\theta_k^*=0$ we have that
$$
\frac{\lambda_\ve}{\ve^2} \sum_{j=1}^q \left( |\theta^*_j + \ve u_j|^p - |\theta_j^*|^p\right) \to
\lambda_0  \sum_{j=1}^q |u_j|^p {\bf 1}_{\{\theta_j^*=0\}}
$$
So, by means of the same arguments adopted in the proof of Theorem \ref{eq:th2}, we can prove that $V_\ve(u) \to_d V(u)$ uniformly on $u$. Following \citet{KimPollard90}, the final step consists in showing that $\arg\min V_\ve = O_{P_{\theta^*}^{(\ve)}}(1)$ and so $\arg\min V_\ve \to_d \arg\min V$.
Indeed,
$$
V_\ve(u) \geq 
\frac{1}{\ve^2} \left(||X - x(\theta^*+\ve u)||^2 - ||X-x(\theta^*)||^2  \right)-
\frac{\lambda_\ve}{\ve^2}  \sum_{j=1}^q  |\ve u_j|^p
$$
and for all $u$ and $\ve$ sufficiently small, $\delta>0$, we have
$$
V_\ve(u) \geq 
\frac{1}{\ve^2} \left(||X - x(\theta^*+\ve u)||^2 - ||X-x(\theta^*)||^2  \right)-
(\lambda_0+\delta)  \sum_{j=1}^q  |u_j|^p=V_\ve^\delta(u).
$$
The term $|u_j|^p$ grows slower than the the first normed terms in  $V_\ve^\delta(u)$, so $\arg\min_u V_\ve^\delta(u) = O_{P_{\theta^*}^{(\ve)}}(1)$ and, in turn, $\arg\min_u V_\ve(u)$ is also $O_{P_{\theta^*}^{(\ve)}}(1)$. Since $\arg\min_u V(u)$ is unique, then the theorem is proved. 
\end{proof}

\begin{remark}
If $\lambda_0=0,$ from the above theorems we immediately obtain that
$$
\ve^{-1}(\hat\theta^\ve - \theta^*) \to_d\arg\min_u V(u)= \mathcal I^{-1} (\theta^*) \zeta,
$$
where $ \mathcal I^{-1} (\theta^*)\zeta\sim N_q({\bf 0},\mathcal I^{-1} (\theta^*) \sigma^2\mathcal I^{-1} (\theta^*)).$
\end{remark}

\section{Adaptive version of the penalized estimator}

Theorem \ref{th3} shows that, if $p<1$, one can estimate the nonzero parameters $\theta_j^*\neq 0$ at the usual rate without introducing asymptotic bias due to the penalization and, at the same time,  shrink the estimates of the null $\theta_j^*=0$ parameters toward zero with positive probability. 

On the contrary, if $p\geq 1$ non zero parameters are estimated with some asymptotic bias if $\lambda_0>0$. This is a well known result in the literature \citep{Zou06} and has been indeed considered in \citet{DegIac12} for ergodic diffusion models with discrete observations. In this section we consider only the case for $p=1$, i.e. the real Lasso estimator.

To state the results we need to rearrange the elements of the vector parameters $\theta$ in this way. Suppose that $q_0\leq q$ values of $\theta^*$ are not null, than we reorder $\theta^*$ as follows: $\theta^*=(\theta_1^*, \ldots, \theta_{q_0}^*,\theta_{q_0+1}^*, \ldots, \theta_{q}^*)^T$, where we denoted by $\theta_k^*=0$, $k=q_0+1, \ldots, q$, the null parameters.
We now need to modify the optimization function by introducing one adaptive sequence for each of the parameters $\theta_j;$ i.e.
\begin{equation}
\label{eq:21}
\tilde Z_\epsilon (u) =  ||X - x(u)|| +   \sum_{j=1}^q \lambda_{\ve,j} |u_j|,
\end{equation}
and, as in the above, 
the \textit{adaptive} Lasso-type estimator is the solution to
\begin{equation}
\tilde\theta^\ve =(\tilde\theta_1^\ve,...,\tilde\theta_q^\ve)= \arg \min_{\theta\in\Theta}  \tilde Z_\ve(\theta).
\label{eq:est21}
\end{equation}
We now need to slightly modify the rate of convergence of the new sequences $\{\lambda_{\ve,j}, j=1, \ldots, q\}$.
\begin{assumption}
\label{ass4}
Let 
$$\kappa_\ve = \min\limits_{j>q_0} \lambda_{\ve, j}\qquad \text{and} \qquad \gamma_\ve = \max\limits_{1\leq j \leq q_0} \lambda_{\ve, j}.$$ 
Then the following convergence must hold
$$\frac{\kappa_\ve}{\ve} \to \infty \qquad \text{and}\qquad  \frac{\gamma_\ve}{\ve} \to 0.$$
\end{assumption}

Let $$\dot x_t^1(\theta)=\left(\frac{\partial}{\partial \theta_1} x_t(\theta),\ldots, \frac{\partial}{\partial \theta_{q_0}} x_t(\theta)\right)^T,$$
and 
\begin{align*}&\mathcal I_{11}(\theta) =\int_0^T \dot x_t^1(\theta) \dot x_t^1(\theta)^T\mu(\de t),\quad (q_0\times 
   q_0 \,\,\text{matrix}).
  \end{align*}
  Let $\eta$ be a Gaussian random vector defined as follows
\begin{equation}
\label{eq:zeta2}
\eta=\int_0^T x^{(1)}_t (\theta^*)\dot x_t^1(\theta^*) \mu(\de t)\sim N_{q_0}({\bf 0},\sigma_1^2),
\end{equation}
where
$$\sigma_1^2=\int_0^T\int_0^T\dot x_t^1(\theta^*)\dot x_s^1(\theta^*)^T E [x^{(1)}_t (\theta^*)x^{(1)}_s (\theta^*)]\mu(\de t) \mu(\de s).$$

The estimator $\tilde\theta^\ve$ reaches asymptotically the oracle properties. Indeed, a good procedure should have the following (asymptotically) properties:
(i) consistently estimates null parameters as zero and vice versa;  i.e. identifies the right subset model;
(ii) has the optimal estimation rate and converges to a Gaussian random variable with covariance matrix of the true subset model.

\begin{theorem}[Oracle properties]
Let Assumptions \ref{exun}--\ref{ass4} hold. Then, as $\ve \to 0$, 
(i) \text{Consistency in variable selection}; i.e. $$P_{\theta^*}^{(\ve)}(\tilde\theta^\ve_k=0) \longrightarrow 1, \quad k=q_0+1, \ldots, q;$$
(ii) \text{Asymptotic normality}; i.e. $$\ve^{-1}(\tilde\theta_1^\ve-\theta_1^*,...,\tilde\theta_{q_0}^\ve-\theta_{q_0}^*)^T \longrightarrow _d \mathcal I_{11}^{-1} (\theta^*)\eta,$$
where $\mathcal I_{11}^{-1} (\theta^*) \eta\sim N_{q_0}({\bf 0},\mathcal I_{11}^{-1}  (\theta^*) \,\sigma_1^2\,\mathcal I_{11}^{-1} (\theta^*)).$
\end{theorem}
\begin{proof}
(i) We briefly outline the proof. The proof is by contradiction. Let us assume that for one $j=q_0+1, \ldots, q$ the adaptive-lasso estimator for $\theta_j^*=0$ is $\tilde\theta^\ve_j\neq 0$. By taking into account the Karush-Kuhn-Tucker (KKT) optimality conditions, we have 
$$
\frac{1}{\ve}\left.\frac{\partial}{\partial u_j} \tilde Z_\epsilon (u) \right|_{u=\tilde\theta^{\ve}} = 
\frac{1}{\ve}\left(\left.\frac{\partial}{\partial u_j}  || X - x(u)|| \right|_{u=\tilde\theta^{\ve}} \right)+ \frac{\lambda_{\ve,j}}{\ve}  {\rm sgn}(\tilde\theta^\ve_j)=0.
$$
The first term is $O_{P_{\theta^*}^{(\ve)}}(1)$ by Assumption \ref{ass1} and the fact that $\tilde\theta^\ve$ is the solution of \eqref{eq:est21}. For the second term we have that $\frac{\lambda_{\ve,j}}{\ve} \geq \frac{\kappa_\ve}{\ve} \to \infty$ by Assumption \ref{ass4}. 

(ii) Let \begin{align}\label{eq:Vad}
\tilde V_\ve(u)& =\frac{1}{\ve^2} \left(||X - x(\theta^*+\ve u)||^2 - ||X-x(\theta^*)||^2 + 
 \sum_{j=1}^q \lambda_{\ve,j}\left\{ |\theta_j^* + \ve u_j| - |\theta_j^*|\right\}\right)\notag\\
 &=u^T ||\dot x(\theta^*)||^2 u -2 u^T ||\ve^{-1}(X-x(\theta^*)) \dot x(\theta^*)||+ o_\ve(1)\notag\\
 &\quad+ \sum_{j=1}^q \frac{\lambda_{\ve,j}}{\ve}\left\{ \frac{|\theta_j^* + \ve u_j| - |\theta_j^*|}{\ve}\right\}
\end{align}
From Assumption \ref{ass4}, since $$u_j\frac{|\theta_j^* + \ve u_j| - |\theta_j^*|}{u_j\ve}\underset{\ve\to 0}{\longrightarrow} u_j sgn(\theta_j^*),$$ for $j=1,...,q_0,$ we have that
\begin{align*}
\sum_{j=1}^{q_0} \frac{\lambda_{\ve,j}}{\ve}\left\{ \frac{|\theta_j^* + \ve u_j| - |\theta_j^*|}{\ve}\right\}\leq \frac{\gamma_\ve}{\ve}\sum_{j=1}^{q_0} \left\{ u_j\frac{|\theta_j^* + \ve u_j| - |\theta_j^*|}{u_j\ve}\right\}\underset{\ve\to 0}{\longrightarrow} 0,
\end{align*}
while for $\theta_j^*=0, j=q_0+1,...,q,$ one has that $\sum_{j=q_0+1}^q \frac{\lambda_{\ve,j}}{\ve} |u_j|\underset{\ve\to 0}{\longrightarrow} \infty.$ Therefore, it is not possible to use the topology of the uniform converge on compact sets; nevertheless, we can define the convergence of $\tilde V_\ve$ via epi-convergence in distribution; i.e. from Lemma 4.1 in \citet{Geyer94}, follows that $\tilde V_\ve(u)\to_d \tilde V(u)$ for every $u,$ where
\begin{align*}
\tilde V(u)=\begin{cases}
u_1^T \mathcal I_{11}(\theta^*) u_1-2u_1^T \eta,& \text{if}\,\, u_{q_0+1}=...=u_q=0,\\
\infty,& \text{otherwise},
\end{cases}
\end{align*}
and $u_1=(u_1,...,u_{q_0})^T$ and the previous convergence is considered on the space of extended functions $\mathbb R^q\to[-\infty,+\infty]$ with a suitable metric. (da fissare meglio) For more details on the epi-convergence see \citet{Geyer94}, \citet{Knight99} and \citet{Rock98}.
Since the unique minimum point of $\tilde V_\ve(u)$ is given by $\ve^{-1}(\tilde \theta^\ve-\theta^*)$ and $\arg\min_u \tilde V(u)=(\mathcal I_{11}^{-1}(\theta^*)\eta,{\bf 0})^T$ is $P_{\theta^*}-$unique, from Theorem 4.4 in \citet{Geyer94} follows the result (ii).

\end{proof}
Now let $\tilde\theta^\ve$ be any consistent estimator of $\theta^*$, for example, the unconstrained minimum distance estimator or the maximum likelihood estimator \citep{Kuto94}. Then, as suggested by \citet{Zou06}, for any constant $\lambda_0>0$ and $\delta>1$, it is sufficient to choose the sequences $\lambda_{\ve, j}$ as follows
\begin{equation}
\lambda_{\ve, j} =\frac{ \lambda_0 }{|\tilde \theta^\ve|^\delta}.
\label{eq:adap}
\end{equation}
If $\lambda_0/\ve \to 0$ and $\ve^{\delta-1}\lambda_0\to \infty$ as $\ve\to 0$, then Assumption \ref{ass4} is satisfied. Usually values of $\delta=1.5$ or $\delta=2$ are common in adaptive Lasso estimation.
The idea of weighting the sequences as in \eqref{eq:adap} is to exploit the ability of consistent estimators to give an initial guess of how large is a parameter, and then using Lasso to shrink adaptively the penalty function in order to avoid bias for true large parameters.

\bibliography{bibliography}

\end{document}